\documentclass[12pt,reqno]{amsart}
\usepackage{amssymb,amsmath,graphicx,
	amsfonts,euscript}
\usepackage{color}
\usepackage{mathrsfs}

\newcommand{\R}{{\mathbb{R}}}
\def\div{ \hbox{\rm div}\,  }
\def\u{ \mathbf{u} }
\def\n{ \mathbf{n} }
\def\b{ \mathbf{B} }
\def\T{ \mathbb{T} }
\newcommand{\Z}{{\mathbb{Z}}}
\newcommand{\N}{{\mathbb{N}}}

\def\nn{\nonumber}

\newtheorem{theorem}{Theorem}[section]

\newtheorem{lemma}[equation]{Lemma}

\theoremstyle{remark}
\newtheorem{remark}[theorem]{Remark}
\numberwithin{equation}{section}
\newcommand{\norm}[2]{\left\lVert #1 \right\rVert_{#2}}

\begin{document}

\title[Global    solutions to  Hall-MHD system  ]{Global  small solutions to the inviscid Hall-MHD system}

\author[Xiaoping Zhai, Yongsheng Li and Yajuan Zhao]{ Xiaoping Zhai$^\dag$, Yongsheng Li$^\ddag$ and Yajuan Zhao$^\ddag$}

\address{$^{\dag}$ School  of Mathematics and Statistics, Shenzhen University,
 Shenzhen, 518060, China}
\email{zhaixp@szu.edu.cn}

\address{$^\ddag$ School of Mathematics,
South China University of Technology,
Guangzhou, 510640, China}
\email{yshli@scut.edu.cn}
\email{zhaoyj\_91@163.com}

\begin{abstract}
The local existence of smooth solutions to the inviscid Hall-MHD system has been obtained
since Chae, Degond and Liu [Ann. Inst. H. Poincar\'e Anal. Non Lin\'eaire,   {31} (2014), 555--565].
However, as we known, how to construct the global small solutions to the  inviscid Hall-MHD  system  is still an open problem. In the present paper, we  give a positive answer  in $\T^3$ when the initial magnetic field is close to a background magnetic field satisfying the Diophantine condition.
	\end{abstract}
\maketitle

\section{ Introduction and main result}
In this paper, we are concerned with the global well-posedness of the smooth solutions to
 the following inviscid Hall-MHD system:
\begin{eqnarray}\label{m1}
\left\{\begin{aligned}
&\partial_t \u+ \u\cdot\nabla \u+\nabla p=(\nabla\times\b)\times\b,\\
&\partial_t \b-\Delta \b-\nabla\times(\u\times\b)=-\nabla\times((\nabla\times\b)\times\b),\\
&\div \u =\div \b =0,\\
&(\u,\b)|_{t=0}=(\u_0,\b_0).
\end{aligned}\right.
\end{eqnarray}
Here, $x=(x_1,x_2,x_3)\in \T^3$ and $t\geq 0$ are the space and time variables, respectively. The unknown $\u$ is the velocity field, $\b$
is the magnetic field, $ {p}$ is the scalar pressure, respectively.

The application of Hall-MHD system is mainly from the understanding of magnetic reconnection phenomena \cite{Homann}, \cite{Lighthill}, where the topology structure of the magnetic field changes dramatically and the Hall effect must be included to get a correct description of this physical process.
Mathematical derivations of Hall-MHD system from either two-fluids or kinetic models can be found in  \cite{Acher}.

The study on the theory of well-posedness of solutions to the Hall-MHD system has grown enormously
in recent years.
For the viscous and resistive Hall-MHD system (with $-\Delta \u$), there have a  lot of excellent works, see instances \cite{Acher}, \cite{chae2014}, \cite{chaejde}, \cite{chaejde2}, \cite{chaejmfm},  \cite{weng}, \cite{danchin}, \cite{Dumas}, \cite{Homann}, \cite{liu}, \cite{Shalybkov}, \cite{wen1}, \cite{wen2}.
More precisely, Chae {\it et al.} \cite{chae2014} showed the global existence of Leray-Hopf weak solutions.
 Dumas {\it et al.} \cite{Dumas}
  have been further investigated  the weak solutions  both for the Maxwell-Landau-Lifshitz system and for the Hall-MHD system. The temporal decay estimates for weak solutions
to Hall-MHD system was established by Chae {\it et al.} \cite{chaejde2}. They also obtained algebraic decay rates for higher
order Sobolev norms of strong solutions with small initial data. It turned out that the Hall term does not affect the time asymptotic behavior, and the time decay rates behaved like those of the corresponding heat equation.
In addition, a blowup criterion and the global
existence of small classical solutions were also established in \cite{chae2014}. These results were later sharpened by
\cite{chaejde}.
In \cite{wen1}, \cite{wen2}, Weng studied the long-time
behaviour and obtained optimal space-time decay rates of strong solutions.

However, there has rather few results about the inviscid Hall-MHD system \eqref{m1}. Chae {\it et al.} \cite{chae2014} obtained the local existence of the smooth solutions in $\R^3$, later Chae {\it et al.} \cite{chaejmfm}
also derived the local  smooth solutions of \eqref{m1} in general dimension by considering the fractional magnetic diffusion.
To the author's knowledge,   it is still an open problem to construct the global  solutions  of the  inviscid Hall-MHD  system even for small initial data.
Inspired by \cite{zhangzhifei}, in the present paper, we  obtain the global small solutions  of the  inviscid Hall-MHD  system in $\T^3$    when  the initial magnetic field is close to a background magnetic field satisfying the Diophantine condition.
Compared with the usual inviscid incompressible MHD system, \eqref{m1} contains the extra term $\nabla\times(\nabla\times \b\times \b)$, which is the so called Hall term. The Hall term heightens the level of nonlinearity of the standard
MHD system from a second-order semilinear to a second-order quasilinear level,
significantly making its qualitative analysis more difficult.

Let ${\n}\in\R^3$ satisfy the so called Diophantine condition: for any $\mathbf{k}\in\Z^3\setminus \{0\},$
\begin{align}\label{diufantu}
|\n\cdot \mathbf{k}|\ge \frac{c}{|\mathbf{k}|^r}, \quad\hbox{for some $c>0$ and $r>2$.}
\end{align}

For the simplicity, we still use the notation $\b$ to denote the perturbation $\b -{\n}$. Hence, the perturbed equations can be rewritten  into
\begin{eqnarray}\label{m}
\left\{\begin{aligned}
&\partial_t \u+ \u\cdot\nabla \u+\nabla p={\n}\cdot\nabla \b+(\nabla\times\b)\times\b,\\
&\partial_t \b-\Delta\b-\nabla\times(\u\times\b)=\mathbf{n}\cdot\nabla \u-\mathbf{n}\cdot\nabla(\nabla\times\b)-\nabla\times((\nabla\times \b)\times\b),\\
&\div \u =\div \b =0,\\
&(\u,\b)|_{t=0}=(\u_0,\b_0).
\end{aligned}\right.
\end{eqnarray}

The   main result of the paper is stated as follows.
\begin{theorem}\label{dingli}
For any ${N}\ge 4r+7$ with $r>2$. Let $(\u_0,\b_0)\in H^{N}(\T^3)$ and
\begin{align}\label{pingjun}
\int_{\T^3}\u_0\,dx=\int_{\T^3}\b_0\,dx=0.
\end{align}
 If there exists a small constant $\varepsilon$ such that
\begin{align*}
\norm{\u_0}{H^{N}}+\norm{\b_0}{H^{N}}\le\varepsilon.
\end{align*}
Then the system \eqref{m} admits a  global  solution $( \u,\b)\in C([0,\infty );H^{N})$. Moreover, for any $t\ge 0$ and $r+4\le\beta <N $,
there holds
\begin{align*}
\norm{\u(t)}{H^{\beta}}+\norm{\b(t)}{H^{\beta}}\le C(1+t)^{-\frac{3({N}-\beta)}{2({N}-r-4)}}.
\end{align*}
\end{theorem}
\begin{remark}
For any $t\ge0$, the solutions of \eqref{m} will preserve this property if there holds \eqref{pingjun}.
\end{remark}
\begin{remark}
The method used here is still valid for the  fractional magnetic diffusion $(-\Delta)^{\gamma}$ with $\gamma>\frac12$. That is to say, we can also extend the local solutions obtained in \cite{chaejmfm} to be global  under the same assumptions of Theorem \ref{dingli}.
\end{remark}
\begin{remark}
Whether can we remove the Diophantine condition \eqref{diufantu} is a  challenged open problem. This is left in the future work.
\end{remark}
\section{The proof of the theorem}
The  proof of the Theorem \ref{dingli} relies heavily on the following lemma whose proof is standard by the Plancheral formula.
\begin{lemma}\label{diu}
If ${\n}\in\R^3$ satisfies the Diophantine condition \eqref{diufantu}, then it holds that for any $s\in R,$
\begin{equation}\label{mor}
\|f\|_{H^{s}}\le C\|{\n}\cdot\nabla f\|_{H^{s+r}}
\end{equation}
if $ \nabla f\in H^{s+r}(\T^3)$ satisfies $\int_{\T^3}f\,dx=0.$
\end{lemma}

Now, we begin to prove the main theorem.
\subsection{$L^2$ energy estimate}
Firstly, denote $\langle a,b\rangle$ the $L^2(\T^3)$ inner product of $a$ and $b$.
A standard energy estimate gives
\begin{align}\label{ping1}
&\frac12\frac{d}{dt}(\norm{\u}{L^2}^2+\norm{\b}{L^2}^2)+\norm{\nabla\b}{L^2}^2=0
\end{align}
where  we have used the   cancellations
\begin{align*}
&\big\langle \u\cdot\nabla \u,\u\big\rangle
   =\big\langle \u\cdot\nabla \b,\b\big\rangle=0,
\quad\big\langle \b\cdot\nabla \b,\u\big\rangle
       +\big\langle \b\cdot\nabla \u,\b\big\rangle=0,\nn\\
&\big\langle{\n}\cdot\nabla \b, \u\big\rangle
   +\big\langle{\n}\cdot\nabla \u, \b \big\rangle=0,\quad
   \big\langle\nabla\times((\nabla\times \b)\times\b),\b\big\rangle=0.
\end{align*}
The above cancellation is crucially important for
 the existence of global smooth solution for small data.
\subsection{High order energy estimate}
  Let $\alpha=(\alpha_{1}, \alpha_{2}, \alpha_{3})\in \N^3$
 be a multi-index.
  We operate
$D^\alpha=\partial^{|\alpha|}/\partial x_1^{\alpha_{1}}\cdots\partial x_3^{\alpha_{3}}$
(where $|\alpha| = \alpha_{1}+\cdots+\alpha_{3}$)
on the first two equations respectively
and take the scalar product of them
with $D^\alpha \u$ and $D^\alpha \b$ respectively,
 add them together and then sum the result over $|\alpha|\leq m$.
  We obtain
\begin{align}\label{3.8}
&\frac{1}{2}\frac{d}{dt}\left(\left\|\u\right\|_{H^m}^2+\left\|\b\right\|_{H^m}^2\right)
        +\left\|\nabla\b\right\|_{H^m}^2\nn\\
&\quad=-\sum_{0<|\alpha|\leq m}\left\langle D^{\alpha}\big((\nabla\times\b)\times\b\big), D^{\alpha}(\nabla\times\b)\right\rangle\nn\\
   &\qquad+\sum_{0<|\alpha|\leq m}\left\langle D^{\alpha}(\u\times\b)
     , D^{\alpha}(\nabla\times\b)\right\rangle\nn\\
&\qquad-\sum_{0<|\alpha|\leq m}\left\langle D^{\alpha}(\u\cdot\nabla \u), D^{\alpha}\u \right\rangle\nn\\
&\qquad+\sum_{0<|\alpha|\leq m}\left\langle D^{\alpha}\left((\nabla\times\b)\times\b\right)
     , D^{\alpha}\u \right\rangle\nn\\
&\quad \stackrel{\mathrm{def}}{=}I_1+I_2+I_3+I_4.
\end{align}
In the following, we estimate successively each of the $I_1$-$I_4$ terms.

For $I_1$, in view of the cancellation
$$\left\langle\big(D^{\alpha}(\nabla\times\b)\big)\times\b,
    D^{\alpha}(\nabla\times\b)\right\rangle=0,$$
 we can get
\begin{align*}
I_1=-\sum_{0<|\alpha|\leq m}\left\langle
    \left[D^{\alpha}\big((\nabla\times\b)\times\b\big)
    -\big(D^{\alpha}(\nabla\times\b)\big)\times\b\right],
    D^{\alpha}(\nabla\times\b)\right\rangle.
\end{align*}
 Using the well-known calculus inequality,
\begin{align}\label{3.9}
 \sum_{|\alpha|\leq m}\left\|D^{\alpha}(fg)-(D^{\alpha}f)g\right\|_{L^2}
 \leq C\big(\|f\|_{H^{m-1}}\|\nabla g\|_{L^{\infty}}
        +\|f\|_{L^{\infty}}\|g\|_{H^{m}}\big),
\end{align}
we have
\begin{align}\label{3.10}
I_1\leq& C\big(\|\b\|_{H^{m}}\|\nabla \b\|_{L^{\infty}}
        +\|\nabla\b\|_{L^{\infty}}\|\b\|_{H^{m}}\big)\left\|\nabla\b\right\|_{H^{m}}\nn\\
  \leq&\frac{1}{4}\left\|\nabla\b\right\|_{H^{m}}^2
         +C \|\b\|^2_{H^{m}}\|\nabla \b\|^2_{L^{\infty}}.
\end{align}
On the other hand, using Leibnitz formula and the Sobolev inequality, we obtain
\begin{align}\label{3.12}
I_2\leq&\sum_{0<|\alpha|\leq m} \left\|D^{\alpha}(\u\times \b)\right\|_{L^2}
          \|\nabla \b\|_{H^{m}}\nn\\
    \leq&C\big(\|\u\|_{L^{\infty}}\|\b\|_{H^{m}}
        +\|\u\|_{H^{m}}\|\b\|_{L^{\infty}}\big)\left\|\nabla\b\right\|_{H^{m}}\nn\\
   \leq&\frac{1}{4}\left\|\nabla\b\right\|_{H^{m}}^2
         +C\|\u\|^2_{L^{\infty}} \|\b\|^2_{H^{m}}
         +C\|\u\|^2_{H^{m}}\|\b\|^2_{L^{\infty}}.
\end{align}
Then, we remark that
\begin{align*}
I_3=-\sum_{0<|\alpha|\leq m}\left\langle
        \big[D^{\alpha}(\u\cdot\nabla \u)-\u\cdot\nabla D^{\alpha}\u
        \big], D^{\alpha}\u \right\rangle.
\end{align*}
Indeed, the second term is zero by the fact that $\u$ is divergence free.
Then, similarly to the above calculation,
using the calculus inequality \eqref{3.9}, we obtain
\begin{align}\label{3.13}
I_3\leq&\sum_{0<|\alpha|\leq m}
        \left\|D^{\alpha}(\u\cdot\nabla \u)-\u\cdot\nabla D^{\alpha}\u\right\|_{L^2}
         \| \u\|_{H^{m}}\nn\\
\leq& C\|\nabla \u\|_{L^\infty}\|\u\|_{H^{m}}^2.
\end{align}
From \eqref{3.8} we get
\begin{align*}
I_4\leq&\sum_{0<|\alpha|\leq m}
        \left\|(\nabla\times\b)\times\b\right\|_{H^{m}}
         \| \u\|_{H^{m}}.
\end{align*}
 Using Leibnitz formula, we derive
\begin{align}\label{3.14}
I_4\leq&C\big(\|\nabla\b\|_{L^{\infty}}\|\b\|_{H^{m}}
        +\|\nabla\b\|_{H^{m}}\|\b\|_{L^{\infty}}\big)\left\| \u\right\|_{H^{m}}\nonumber\\
\leq&C\|\nabla\b\|_{L^{\infty}}\|\b\|_{H^{m}}\left\| \u\right\|_{H^{m}}
         +\frac{1}{2}\|\nabla\b\|_{H^{m}}^2
         +C\|\b\|_{L^{\infty}}^2\left\|\u\right\|_{H^{m}}^2\nonumber\\
         \leq&\frac{1}{2}\|\nabla\b\|_{H^{m}}^2+C(\|\nabla\b\|_{L^{\infty}}+\|\b\|_{L^{\infty}}^2)
         (\left\|\u\right\|_{H^{m}}^2+\left\|\b\right\|_{H^{m}}^2).
\end{align}
From estimates \eqref{3.10}, \eqref{3.12}, \eqref{3.13} and \eqref{3.14}, we obtain
\begin{align}\label{3.16}
\frac{d}{dt}&\big(\|\u\|_{H^{m}}^2+\|\b\|_{H^{m}}^2\big)
 +\|\nabla\b\|_{H^{m}}^2
       \nonumber\\
 \leq&C\big(\|\nabla\u\|_{L^{\infty}}+\|\nabla\b\|_{L^{\infty}}+\|\nabla\b\|_{L^{\infty}}^2+\|\u\|_{L^{\infty}}^2+\|\b\|_{L^{\infty}}^2\big)
  \big(\|\u\|_{H^{m}}^2+\|\b\|_{H^{m}}^2\big).
\end{align}

\subsection{A key Lemma}
The following lemma which relies heavily on the structural characteristics of the system \eqref{m} is crucial to get the time decay for the velocity field.
\begin{lemma}\label{ping13}
Assume that
\begin{align}\label{ping14}
\sup_{t\in[0,T]}(\norm{\u}{H^{N}}+\norm{\b}{H^{N}})\le \delta,
\end{align}
for some $0<\delta<1.$ Then there holds that
\begin{align}\label{ping15}
&\norm{{\n}\cdot\nabla \u}{H^{r+3}}^2-\sum_{0\le s\le r+3}\frac{d}{dt}\big\langle{D^s}\b,{D^s}({\n}\cdot\nabla \u)\big\rangle\nn\\
&\quad\le C\norm{ \b}{H^{r+5}}^2+C\delta^2\norm{\u}{H^3}^2.
\end{align}

\end{lemma}

\begin{proof}
Applying ${D^s} (0\le s\le r+3) $ to the second equation of \eqref{m}, and multiplying it by ${D^s}({\n}\cdot\nabla \u)$
then integrating over $\T^3$, we obtain
\begin{align}\label{ping16}
\norm{{D^s}({\n}\cdot\nabla \u)}{L^2}^2
=&{\big\langle{D^s}\partial_t \b,{D^s}({\n}\cdot\nabla \u)\big\rangle}-{\big\langle{D^s} \Delta\b,{D^s}({\n}\cdot\nabla \u)\big\rangle}\nn\\
&+{\big\langle{D^s} (\u\cdot\nabla \b),{D^s}({\n}\cdot\nabla \u)\big\rangle}\nn\\
&-{\big\langle{D^s}(\b\cdot\nabla \u),{D^s}({\n}\cdot\nabla \u)\big\rangle}\nn\\
&+{\big\langle{D^s} (\mathbf{n}\cdot\nabla(\nabla\times\b)),{D^s}({\n}\cdot\nabla \u)\big\rangle}\nn\\
&+{\big\langle{D^s} (\nabla\times((\nabla\times \b)\times\b)),{D^s}({\n}\cdot\nabla \u)\big\rangle}\nn\\
 \stackrel{\mathrm{def}}{=}&I_5+I_6+I_7+I_8+I_9+I_{10}.
\end{align}
Thanks to the H\"older inequality, Young's inequality, and the embedding relation,
 we have
\begin{align}\label{ping16+1}
I_6\le& C\norm{{D^s}\Delta\b}{L^2}\norm{{D^s}({\n}\cdot\nabla \u)}{L^2}\nn\\
\le&\frac{1}{16}\norm{{D^s}({\n}\cdot\nabla \u)}{L^2}^2+C\norm{\Delta \b}{H^s}^2\nn\\
\le&\frac{1}{16}\norm{{D^s}({\n}\cdot\nabla \u)}{L^2}^2+C\norm{\b}{H^{s+2}}^2\nn\\
\le&\frac{1}{16}\norm{{D^s}({\n}\cdot\nabla \u)}{L^2}^2+C\norm{\b}{H^{r+5}}^2.
\end{align}
Similarly,  using Leibnitz formula and the Sobolev inequality, we obtain
\begin{align}\label{ping17}
I_7\le& C\norm{{D^s}(\u\cdot\nabla \b)}{L^2}\norm{{D^s}({\n}\cdot\nabla \u)}{L^2}\nn\\
\le& C(\norm{\u}{L^\infty}\norm{\nabla \b}{H^s}+\norm{\nabla \b}{L^\infty}\norm{ \u}{H^s} )\norm{{D^s}({\n}\cdot\nabla \u)}{L^2}\nn\\
\le&\frac{1}{16}\norm{{D^s}({\n}\cdot\nabla \u)}{L^2}^2+C(\norm{\u}{H^2}^2\norm{\nabla \b}{H^s}^2+\norm{\nabla \b}{H^2}^2\norm{ \u}{H^s}^2 )\nn\\
\le&\frac{1}{16}\norm{{D^s}({\n}\cdot\nabla \u)}{L^2}^2+C\norm{\u}{H^N}^2\norm{ \b}{H^{s+1}}^2\nn\\
\le&\frac{1}{16}\norm{{D^s}({\n}\cdot\nabla \u)}{L^2}^2+C\delta^2\norm{ \b}{H^{r+4}}^2,
\end{align}
and
\begin{align}\label{ping19}
I_8\le& C\norm{{D^s}(\b\cdot\nabla \u)}{L^2}\norm{{D^s}({\n}\cdot\nabla \u)}{L^2}\nn\\
\le& C(\norm{\b}{L^\infty}\norm{\nabla \u}{H^s}+\norm{\nabla \u}{L^\infty}\norm{ \b}{H^s} )\norm{{D^s}({\n}\cdot\nabla \u)}{L^2}\nn\\
\le& C(\norm{\b}{H^2}\norm{ \u}{H^{s+1}}+\norm{\nabla \u}{H^2}\norm{ \b}{H^s} )\norm{{D^s}({\n}\cdot\nabla \u)}{L^2}\nn\\
\le&\frac{1}{16}\norm{{D^s}({\n}\cdot\nabla \u)}{L^2}^2+C\norm{\u}{H^N}^2(\norm{ \b}{H^{s}}^2+\norm{ \b}{H^{2}}^2)\nn\\
\le&\frac{1}{16}\norm{{D^s}({\n}\cdot\nabla \u)}{L^2}^2+C\delta^2\norm{ \b}{H^{r+3}}^2.
\end{align}
The estimate $I_9$ is similar to $I_6$,
\begin{align}\label{ping18}
I_9\le& C\norm{{D^s}\nabla^2\b}{L^2}\norm{{D^s}({\n}\cdot\nabla \u)}{L^2}\nn\\
\le&\frac{1}{16}\norm{{D^s}({\n}\cdot\nabla \u)}{L^2}^2+C\norm{\nabla^2 \b}{H^s}^2\nn\\
\le&\frac{1}{16}\norm{{D^s}({\n}\cdot\nabla \u)}{L^2}^2+C\norm{\b}{H^{s+2}}^2\nn\\
\le&\frac{1}{16}\norm{{D^s}({\n}\cdot\nabla \u)}{L^2}^2+C\norm{\b}{H^{r+5}}^2.
\end{align}
Moreover,
for any $0\le s\le r+3$ and ${N}\ge r+4$, there holds
\begin{align}\label{ping20}
I_{10}\le&\big\langle{D^s} (\b\cdot\nabla(\nabla\times\b)-(\nabla\times\b)\cdot\nabla\b),{D^s}({\n}\cdot\nabla \u)\big\rangle\nn\\
\le& C(\norm{\b}{L^\infty}\norm{\nabla^2 \b}{H^s}+\norm{\nabla^2 \b}{L^\infty}\norm{ \b}{H^s} +\norm{\nabla \b}{L^\infty}\norm{ \nabla\b}{H^s} )\norm{{D^s}({\n}\cdot\nabla \u)}{L^2}\nn\\
\le&\frac{1}{16}\norm{{D^s}({\n}\cdot\nabla \u)}{L^2}^2+C\norm{ \b}{H^{s+2}}^2\norm{\b}{H^N}^2\nn\\
\le&\frac{1}{16}\norm{{D^s}({\n}\cdot\nabla \u)}{L^2}^2+C\delta^2\norm{\b}{H^{r+5}}^2.
\end{align}
Finally, we have to bound the first term on the
right hand side of \eqref{ping16}. In fact, exploiting the first equation in \eqref{m}, we can rewrite this term into
\begin{align}\label{ping21}
\big\langle{D^s}\partial_t \b,{D^s}({\n}\cdot\nabla \u)\big\rangle
=&\frac{d}{dt}\big\langle{D^s}\b,{D^s}({\n}\cdot\nabla \u)\big\rangle-\big\langle{D^s}\b,{D^s}({\n}\cdot\nabla \partial_t\u)\big\rangle\nn\\
=&\frac{d}{dt}\big\langle{D^s}\b,{D^s}({\n}\cdot\nabla \u)\big\rangle+\big\langle{D^s}({\n}\cdot\nabla\b),{D^s} \partial_t\u\big\rangle\nn\\
=&\frac{d}{dt}\big\langle{D^s}\b,{D^s}({\n}\cdot\nabla \u)\big\rangle+\big\langle{D^s}({\n}\cdot\nabla \b),{D^s}(\b\cdot\nabla \b)\big\rangle\\
&+\big\langle{D^s}({\n}\cdot\nabla \b),{D^s}({\n}\cdot\nabla \b)\big\rangle-\big\langle{D^s}({\n}\cdot\nabla \b),{D^s}(\u\cdot\nabla \u)\big\rangle\nn
\end{align}
where we have used the following cancellation
\begin{align*}
\big\langle{D^s}({\n}\cdot\nabla\b),{D^s} \nabla p\big\rangle=0.
\end{align*}
The second term on the right hand side of \eqref{ping21} can be bounded as
\begin{align}\label{pingping2}
\big\langle{D^s}({\n}\cdot\nabla \b),{D^s}(\b\cdot\nabla \b)\big\rangle
\le& C\norm{{D^s}({\n}\cdot\nabla \b)}{L^2}(\norm{\b}{L^\infty}\norm{\nabla \b}{H^s}+\norm{\nabla \b}{L^\infty}\norm{ \b}{H^s} )\nn\\
\le& C\norm{\b}{H^N}\norm{\nabla \b}{H^s}^2\nn\\
\le& C\delta\norm{\nabla \b}{H^s}^2.
\end{align}
In the same manner, we can deal with the last term  on the right hand side of \eqref{ping21}
\begin{align}\label{pingping3}
\big\langle{D^s}({\n}\cdot\nabla \b),{D^s}(\u\cdot\nabla \u)\big\rangle
\le& C\norm{{D^s}({\n}\cdot\nabla \b)}{L^2}(\norm{\u}{L^\infty}\norm{\nabla \u}{H^s}+\norm{\nabla \u}{L^\infty}\norm{ \u}{H^s} )\nn\\
\le& C\norm{{D^s}({\n}\cdot\nabla \b)}{L^2}(\norm{\u}{H^2}\norm{\nabla \u}{H^s}+\norm{\nabla \u}{H^2}\norm{ \u}{H^s} )\nn\\
\le& C\norm{\nabla \b}{H^s}\norm{\u}{H^N}\norm{\u}{H^3}\nn\\
\le& C\norm{\nabla \b}{H^s}^2+C\delta^2\norm{\u}{H^3}^2.
\end{align}
Inserting \eqref{pingping2} and \eqref{pingping3} into \eqref{ping21} gives
\begin{align}\label{ping22}
\big\langle{D^s}\partial_t \b,{D^s}({\n}\cdot\nabla \u)\big\rangle
\le&\frac{d}{dt}\big\langle{D^s}\b,{D^s}({\n}\cdot\nabla \u)\big\rangle+C\norm{ \nabla\b}{H^{s}}^2+C\delta^2\norm{ \u}{H^{3}}^2.
\end{align}
Plugging  \eqref{ping16+1}--\eqref{ping18} and \eqref{ping22} into \eqref{ping16},
we can arrive at \eqref{ping15}. This prove the lemma.
\end{proof}

\subsection{Complete the proof of the main theorem}
Given the initial data $(\u_0, \b_0)\in H^{N}$, the local well-posedness of the system \eqref{m} has been  proved in \cite{chae2014}
by using the energy method. Thus, we may assume that there exist $T > 0$ and a unique solution
$(\u,\b)\in C([0,T];H^{N})$ of the system \eqref{m}. Furthermore, we may assume that
\begin{align}\label{ping23}
\sup_{t\in[0,T]}(\norm{\u}{H^{N}}+\norm{\b}{H^{N}})\le \delta,
\end{align}
for some $0<\delta<1$
 to be determined later.

According to the embedding relation, we get for any $N\ge 3$ that
\begin{align}\label{gan1}
&\|\nabla\u\|_{L^{\infty}}+\|\nabla\b\|_{L^{\infty}}+\|\nabla\b\|_{L^{\infty}}^2+\|\u\|_{L^{\infty}}^2+\|\b\|_{L^{\infty}}^2\nn\\
&\quad\le C(\|\nabla\u\|_{H^{2}}+\|\nabla\b\|_{H^{2}}+\|\nabla\b\|_{H^{2}}^2+\|\u\|_{H^{2}}^2+\|\b\|_{H^{2}}^2)\nn\\
&\quad\le C(\|\u\|_{H^{N}}+\|\b\|_{H^{N}}+\|\u\|_{H^{N}}^2+\|\b\|_{H^{N}}^2)\nn\\
&\quad\le C\delta(1+\delta)
\end{align}
from which and taking $m=r+4$ in \eqref{3.16} gives
 \begin{align}\label{gan2}
\frac12\frac{d}{dt}&\big(\|\u\|_{H^{r+4}}^2+\|\b\|_{H^{r+4}}^2\big)
 +\|\nabla\b\|_{H^{r+4}}^2
 \leq C\delta(1+\delta) \big(\|\u\|_{H^{r+4}}^2+\|\b\|_{H^{r+4}}^2\big).
\end{align}
Due to
$$\int_{\T^3}\b_0\,dx=0,$$
there holds
\begin{align}\label{gan3}
\|\b\|_{H^{r+5}}^2\le C\|\nabla\b\|_{H^{r+4}}^2.
\end{align}
Hence, let
 $A\ge 1+2C$ be a constant determined later, we infer from Lemma \ref{ping13},  \eqref{gan2} and \eqref{gan3} that
\begin{align}\label{ping24}
&\frac{d}{dt}\left\{A(\norm{\u}{H^{r+4}}^2+\norm{\b}{H^{r+4}}^2)-\sum_{0\le s\le r+3}\big\langle{D^s}\b,{D^s}({\n}\cdot\nabla \u)\big\rangle\right\}\nn\\
&\quad+A\norm{\nabla\b}{H^{r+4}}^2+\norm{{\n}\cdot\nabla \u}{H^{r+3}}^2\nn\\
&\le CA\delta(1+\delta) \big(\|\u\|_{H^{r+4}}^2+\|\b\|_{H^{r+4}}^2\big)
+C\delta^2\norm{\u}{H^3}^2.
\end{align}

In view of Lemma \ref{diu}, for any ${N}\ge 2r+5$,  we have
\begin{align}\label{ping25}
\norm{ \u}{H^{3}}^2\le& C\norm{{\n}\cdot\nabla \u}{H^{r+3}}^2\quad
\norm{ \u}{H^{r+4}}^2\le\norm{ \u}{H^{3}}\norm{ \u}{H^{{N}}}\le C\delta\norm{{\n}\cdot\nabla \u}{H^{r+3}}
\end{align}
which gives
\begin{align}\label{}
&CA\delta(1+\delta) \|\b\|_{H^{r+4}}^2
  \le CA\delta^2(1+\delta)\norm{ \nabla\b}{H^{r+4}},
\end{align}
and
\begin{align}\label{}
&CA\delta(1+\delta) \|\u\|_{H^{r+4}}^2+C\delta^2\norm{\u}{H^3}^2
  \le CA\delta^2(1+\delta)\norm{{\n}\cdot\nabla \u}{H^{r+3}}.
\end{align}

As a result, we can infer from \eqref{ping24} that
\begin{align}\label{ping30}
&\frac{d}{dt}\left\{A(\norm{\u}{H^{r+4}}^2+\norm{\b}{H^{r+4}}^2)-\sum_{0\le s\le r+3}\big\langle{D^s}\b,{D^s}({\n}\cdot\nabla \u)\big\rangle\right\}\nn\\
&\qquad+A\norm{\nabla\b}{H^{r+4}}^2+\norm{{\n}\cdot\nabla \u}{H^{r+3}}^2\nn\\
&\quad\le CA\delta^2(1+\delta)\norm{{\n}\cdot\nabla \u}{H^{r+3}}+CA\delta^2(1+\delta)\norm{ \nabla\b}{H^{r+4}}.
\end{align}
Define
\begin{align*}
{\mathcal{D}(t)}=&A\norm{\nabla\b}{H^{r+4}}^2+\norm{{\n}\cdot\nabla \u}{H^{r+3}}^2,\nn\\
{\mathcal{E}(t)}=&A(\norm{\u}{H^{r+4}}^2+\norm{\b}{H^{r+4}}^2)-\sum_{0\le s\le r+3}\big\langle{D^s}\b,{D^s}({\n}\cdot\nabla \u)\big\rangle.
\end{align*}
Taking $A>1$ such that
$${\mathcal{E}(t)}\ge\norm{\u}{H^{r+4}}^2+\norm{\b}{H^{r+4}}^2.$$
Hence, by choosing $\delta>0$ small enough, we can get that
\begin{align}\label{ping33}
\frac{d}{dt}{\mathcal{E}(t)}+\frac12{\mathcal{D}(t)}\le 0.
\end{align}
For any ${N}\ge 4r+7$, by the interpolation inequality, we have
\begin{align*}
\norm{ \u}{H^{r+4}}^2\le&\norm{ \u}{H^{3}}^{\frac32}\norm{ \u}{H^{{N}}}^{\frac12}\le C\delta^{\frac12}\norm{{\n}\cdot\nabla \u}{H^{r+3}}^{\frac32}
\end{align*}
which further implies that
\begin{align*}
{\mathcal{E}(t)}\le& C(\norm{\u}{H^{r+4}}^2+\norm{\b}{H^{r+4}}^2)\nn\\
\le& C\norm{ \u}{H^{3}}^{\frac32}\norm{ \u}{H^{{N}}}^{\frac12}+C\norm{ \b}{H^{3}}^{\frac32}\norm{ \b}{H^{{N}}}^{\frac12}\nn\\
\le& C\delta^{\frac12}\norm{{\n}\cdot\nabla \u}{H^{r+3}}^{\frac32}+C\delta^{\frac12}\norm{\nabla \b}{H^{r+4}}^{\frac32}\nn\\
\le&(D(t))^{\frac34}.
\end{align*}
So, we get a  Lyapunov-type inequality
\begin{align*}
\frac{d}{dt}{\mathcal{E}(t)}+c({\mathcal{E}(t)})^{\frac43}\le 0.
\end{align*}
Solving this inequality yields
\begin{align}\label{ping37}
{\mathcal{E}(t)}\le C(1+t)^{-3}.
\end{align}
Taking $m={N}$ in \eqref{3.16} and using the embedding relation give
\begin{align}\label{ping38}
&\frac{d}{dt}(\norm{\u}{H^{N}}^2+\norm{\b}{H^{N}}^2)+\norm{\nabla\b}{H^{N}}^2\nn\\
&\quad\le C(\norm{ \u}{H^3}+\norm{ \u}{H^2}^2+\norm{\b}{H^3}+\norm{\b}{H^3}^2)(\norm{\u}{H^{N}}^2+\norm{\b}{H^{N}}^2).
\end{align}
From \eqref{ping37}, we have
\begin{align}\label{}
\int_0^t(\norm{ \u(\tau)}{H^3}+\norm{\u(\tau)}{H^2}^2+\norm{ \b(\tau)}{H^3}+\norm{\b(\tau)}{H^3}^2)\,d\tau\le C,
\end{align}
thus, exploiting
the Gronwall inequality implies
\begin{align}\label{ping39}
\norm{\u}{H^{N}}^2+\norm{\b}{H^{N}}^2
\le&C(\norm{\u_0}{H^{N}}^2+\norm{\b_0}{H^{N}}^2)
\le C\varepsilon^2.
\end{align}
Taking $\varepsilon$ small enough so that $C\varepsilon\le \delta/2$, we deduce from a continuity argument that the local solution
can be extended as a global one in time.

Moreover, from \eqref{ping37}, we also have the following decay rate
\begin{align*}
\norm{\u(t)}{H^{r+4}}+\norm{\b(t)}{H^{r+4}}\le C(1+t)^{-\frac{3}{2}}.
\end{align*}
Thus,
for any $\beta> r+4$, choosing ${N}>\beta$ and  using the following interpolation inequality
\begin{align*}
\norm{f(t)}{H^{\beta}}\le\norm{f(t)}{H^{r+4}}^{\frac{{N}-\beta}{{N}-r-4}}
\norm{f(t)}{H^{N}}^{\frac{\beta-r-4}{{N}-r-4}},
\end{align*}
we can get the decay rate for the higher order energy
\begin{align*}
\norm{\u(t)}{H^{\beta}}+\norm{\b(t)}{H^{\beta}}\le C(1+t)^{-\frac{3({N}-\beta)}{2({N}-r-4)}}.
\end{align*}
This completes the proof of Theorem \ref{dingli}.$\hspace{8.3cm}\square$

\section*{ Acknowledgments}
This work is  supported by  the National Natural Science Foundation of China under grant number 11601533,  the
National Natural Science Foundation key project of China under grant number 11831003, and the
Science and Technology Program of Shenzhen under grant number 20200806104726001.

\end{document}